\documentclass[12pt,leqno]{article}
\usepackage{amsmath,amssymb,amsfonts,amsthm,latexsym,amstext,amscd}
\usepackage[english]{babel}
\usepackage{fancyhdr}

\newtheorem{theorem}{Theorem}
\newtheorem{corollary}{Corollary}
\newtheorem{lemma}{Lemma}

\newtheorem*{prop}{Proposition}

\theoremstyle{definition}

\newcommand{\beq}{\begin{equation}}
\newcommand{\eeq}{\end{equation}}

\def\ve{\varepsilon }
\begin{document}

\title
{Some new density theorems for Dirichlet $L$-functions}

\author
{J\'anos Pintz\thanks{Supported by ERC-AdG.~321104 and National Research Development and Innovation Office, NKFIH, K~119528.}
}

\date{\it Dedicated to the 60\,\textsuperscript{th} birthday of Jerzy Kaczorowski}

\numberwithin{equation}{section}


\maketitle


{\bf 1.} It is well known that the maximum of the error term of the Prime Number Theorem (or its analogue for primes in arithmetic progression) depends on the zero-free region of the Riemann's zeta-function (or the Dirichlet's $L$-functions, respectively) or in  other words on the zeros lying nearest to the boundary line $\text{\rm Re } s = \sigma = 1$.

On the other hand, many other arithmetic problems depend not only on the situation of the extreme right hand zeros but also on the number of such zeros.
The first theorem of such type was proved nearly 100 years ago by F. Carlson \cite{Car} in 1920.
These results were called later density theorems.
They proved to be very useful in bounding from above gaps between consecutive primes (or between consecutive primes in an arithmetic progression) or in Linnik's problem of bounding the first prime in an arithmetic progression.

In some applications the distribution of all zeros $\varrho = \beta + i\gamma$ with $1/2 \leq \beta \leq 1$ is important (like in the case of bounding from above gaps between consecutive primes).
In other questions only those lying near to the line $\sigma = 1$, i.e. with $\beta \geq 1 - \ve$ ($\ve$ small) play a significant role (like in case of Linnik's constant).
In the second type problems it is also important to prove so-called ``log-free'' density theorems, where the upper bound for the zeros does not contain any power of the logarithm of the modulus of the relevant arithmetic progression or the logarithm of the height of the relevant zeros.

\setcounter{section}{2}
\medskip
{\bf 2.}
In Section~3 we will prove density theorems in the half-planes $\sigma > 3/4$ (Theorem~\ref{th:1}) and $\sigma > 1 - \ve$ (Theorems \ref{th:2}--\ref{th:3}), respectively.
In the following introduction we will focus on density theorems near the line $\sigma = 1$.
Let $N(\alpha, T, \chi)$ denote the number of zeros of $L(s, \chi)$ in the rectangle
\beq
\label{eq:2.1}
R(\alpha, T) = \{\sigma + it; \ \alpha \leq \sigma \leq 1, \ |t| \leq T\}
\eeq
and let
\begin{align}
\label{eq:2.2}
N(\alpha, T, q) &= \sum_{\chi(\text{\rm mod }q)} N(\alpha, T, \chi),\\
\label{eq:2.3}
N^*(\alpha, T, Q) &= \sum_{q \leq Q} \underset{\chi(\text{\rm mod }q)}{\sum\nolimits^\ast} N(\alpha, T, \chi),
\end{align}
where the asterisk indicates summation over primitive characters.
Fogels \cite{Fog} and Gallagher \cite{Gal} proved the first general ``log-free'' density theorems of the form
\beq
\label{eq:2.4}
N(\alpha, T, q) \ll T^{c(1 - \alpha)} \ \ \text{ for fixed } \ q \leq T \ \ \ \text{\cite{Gal}}
\eeq
uniformly and
\beq
\label{eq:2.5}
N^\ast(\alpha, T, T) \ll T^{c'(1 - \alpha)} \ \ \ \text{\cite{Fog}}
\eeq
with large values of $c$ and~$c'$.

Selberg invented a new method -- the use of the so-called pseudocharacters (cf. \eqref{eq:3.1}) -- which yielded the estimates
\cite{MS}:
\begin{align}
\label{eq:2.6}
N(\alpha, T, q) &\ll_\ve(q^{c_1} T^{c_2})^{(1 + \ve)(1 - \alpha)},\\
\label{eq:2.7}
N^*(\alpha, T, Q) &\ll_\ve(Q^{c_3} T^{c_4})^{(1 + \ve)(1 - \alpha)},
\end{align}
with $c_1 = c_2 = 3$, $c_3 = 5$, $c_4 = 3$.
This was improved later by Motohashi \cite{Mot}
(for $4/5 \leq \alpha \leq 1$) to $c_1 = 2$, $c_2 = 3$, $c_3 = 4$, $c_4 = 3$ and Jutila \cite{Jut} to $c_1 = c_2 = 2$, $c_3 = 4$, $c_4 = 2$ (for $4/5 \leq \alpha \leq 1$).
Jutila mentioned that, according to a remark of Huxley, the value of $c_3$ can be improved to $3$ if $\alpha$ is near to $1$ $\bigl(1 - c(\varepsilon) \leq \alpha \leq 1\bigr)$.
This will be denoted by $c_3' = 3$.
Wang \cite{Wan} showed with this notation $c_1' = c_2' = 3/2$, $c_3' = 3$, $c_4' = 3/2$.

In what follows we will use a method of S. W. Graham \cite{Gra2} and Heath-Brown \cite{Hea3} to improve this to $c_3' = 9/4$ and to obtain many new results.
Our method will give estimates for $N(\alpha, T, \chi)$ with individual characters $\chi \text{\rm mod } q$ as, for example
\beq
\label{eq:2.8}
N(\alpha, T, \chi_q) \ll_\ve (qT)^{(3/4 + \ve)(1 - \alpha)} \ \ \ (1 - \ve^3 \leq \alpha \leq 1).
\eeq
In particular, in case of $q = 1$, we obtain results for the number of zeros $N(\alpha, T)$ of $\zeta(s)$.

In later applications we will often need density theorems for a given subset of all primitive characters with moduli $\leq Q$ instead of \eqref{eq:2.7}.
Since the method applied in Section~3 yields much better results for these subsets, we will formulate our results in this more general setting.
We will then obtain estimates of type \eqref{eq:2.6}--\eqref{eq:2.8} as immediate consequences of the general theorem.
In the following, let $\text{\rm cond }\chi$ denote the conductor of~$\chi$.

\begin{theorem}
\label{th:1}
Let $\mathcal H$ be a set of primitive characters $\chi$ with moduli $\leq M$ such that $\text{\rm cond }\chi_i \overline{\chi}_j \leq K$ for any pair $\chi_i, \chi_j$ belonging to~$\mathcal H$.
Let $\mathcal S$ be a set of distinct pairs $(\chi_j, \varrho_j)$ with $L(\varrho_j \chi_j) = 0$, where $\chi_j \in \mathcal H$, $\varrho_j \in R(\alpha, T)$.
(The same character might naturally appear in $\mathcal S$ several times with different zeros.)
Let $J$ denote the cardinality of $\mathcal S$ and let $\ve$ be an arbitrary, sufficiently small positive number $(0 < \ve < c_0)$. Then for $\alpha > 4/5$, $T \geq 3$, we have
\beq
\label{eq:2.9}
J \ll_\ve \left( M^{\frac{1}{2\alpha - 1}} K^{\frac{1}{(2\alpha - 1)(4\alpha - 3)}} T^{\frac{2}{4\alpha - 3}}\right)^{\left(\frac34 + \ve\right)(1 - \alpha)}.
\eeq
\end{theorem}

\begin{corollary}
\label{cor:1}
The following estimates hold for $\alpha > 4/5$, $T \geq 3$:
\beq
\label{eq:2.10}
\aligned
N^*(\alpha, T, Q) &\ll_\ve \left(Q^{\frac{(3 + \ve)(4\alpha - 1)}{4(4\alpha - 3)(2\alpha - 1)}}
T^{\frac{3 + \ve}{2(4\alpha - 3)}}\right)^{(1 - \alpha)},\\
N(\alpha, T, q) &\ll_\ve (qT)^{\frac{(3 + \ve)(1 - \alpha)}{2(4\alpha - 3)}},\\
N(\alpha, T, \chi_q) &\ll_\ve \left(q^{\frac{1}{2\alpha - 1}} T^{\frac{2}{4\alpha - 3}}\right)^{\left(\frac34 + \ve\right)(1 - \alpha)},\\
N(\alpha, T) &\ll_\ve T^{\frac{(3 + \ve)(1 - \alpha)}{2(4\alpha - 3)}}.
\endaligned
\eeq
\end{corollary}

We remark that while the following estimate for $N^*(\alpha, T, Q)$ is distinctly sharper than that of Wang \cite{Wan}, the one for $N(\alpha, T, q)$ is just slightly better and for $\alpha \to 1$ asymptotically equal.
Further, \eqref{eq:2.10} is sharper than the density hypothesis in the $q$ and $T$ aspect for $\alpha > \frac{15}{16}$ while
is sharper in the $Q$ aspect for $\alpha > 0.9020456\dots$.
(The estimates of Wang \cite{Wan} are better than the density theorem in case of $\alpha > 23/24$ in all aspects.)

If at least one of $K$ and $M$ is small, the following result (Theorem~\ref{th:2}) will be of interest.

\begin{theorem}
\label{th:2}
Under the conditions of Theorem~\ref{th:1} for $\alpha > 1 - \ve^3$, $T \geq 3$ we have
\begin{align}
\label{eq:2.10A}
J &\ll_\ve (K^2(MT)^{3/4})^{(1 + \ve)(1 - \alpha)},\\
\label{eq:2.10B}
J &\ll_\ve (M^2(KT)^{3/4})^{(1 + \ve)(1 - \alpha)},\\
\label{eq:2.10C}
J &\ll_\ve (M^2 K^2 T^{2\ve})^{(1 + \ve)(1 - \alpha)}.
\end{align}
\end{theorem}

\begin{corollary}
\label{cor:2}
For $\alpha > 1 - \ve^3$, $T \geq 3$ we have
\begin{align}
\label{eq:2.11}
N(\alpha, T, \chi_q) &\ll_\ve (qT)^{(3/4 + \ve)(1 - \alpha)},\\
\label{eq:2.11A}
N(\alpha, T, \chi_q) &\ll_\ve (q^2 T^{2\ve})^{(1 + \ve)(1 - \alpha)},\\
\label{eq:2.11B}
N(\alpha, T, q) &\ll_\ve (q^4 T^{2\ve})^{(1 + \ve)(1 - \alpha)},\\
\label{eq:2.11C}
N^*(\alpha, T, Q) &\ll_\ve (Q^6 T^{2\ve})^{(1 + \ve)(1 - \alpha)},\\
\label{eq:2.11D}
N(\alpha, T) &\ll_\ve T^{\ve(1 - \alpha)}.
\end{align}
\end{corollary}

On the other hand, if $T$ is much smaller than $q$ or $Q$ (or it is bounded, for example, as in the proof of Linnik's theorem), then the following generalization of Heath-Brown's Lemma 11.1 \cite{Hea3} leads to improvements over Theorem~\ref{th:1}.

\begin{theorem}
\label{th:3}
Suppose the conditions of Theorem~\ref{th:1} and let $\varphi = 1/4$ if all characters in $\mathcal H$ have cube-free moduli (or order at most $\log M$), otherwise let $\varphi = 1/3$.
Then for $\alpha > 1 - \ve^2$, $T \geq 3$ we have
\beq
\label{eq:2.12}
J \ll_\ve \bigl((KM)^{2\varphi + \ve} T^{10/\ve}\bigr)^{(1 - \alpha)}.
\eeq
\end{theorem}

\begin{corollary}
\label{cor:3}
With the notation of Theorem~\ref{th:3} we have for $\alpha > 1 - \ve^2$, $T \geq 3$
\begin{align}
N^*(\alpha, T, Q) &\ll_\ve (Q^{2 + \ve} T^{10/\ve})^{1 - \alpha}, \nonumber\\
\label{eq:2.13}
N(\alpha, T, q) &\ll_\ve \begin{cases}
(q^{1 + \ve} T^{10/\ve})^{(1 - \alpha)} &\text{if $q$ is cube-free},\\
(q^{4/3 + \ve} T^{10/\ve})^{(1 - \alpha)} &\text{otherwise}.
\end{cases}
\end{align}
\end{corollary}

If $K$ and $M$ are both significantly smaller than $T$, then the results \eqref{eq:2.10C}, \eqref{eq:2.11A}--\eqref{eq:2.11D} are much better than \eqref{eq:2.10}, \eqref{eq:2.10A} and \eqref{eq:2.10B}.
The first results having an expression of the type $T^{o(1 - \alpha)}$ for $\alpha \to 1$ were proved by Hal\'asz and Tur\'an \cite{HT}.

Finally, K. Ford \cite{For} showed, as a consequence of his explicit estimate for the zeta-function (cf. \eqref{eq:3.4} in our next section) the inequality
\beq
\label{eq:2.16}
N(\alpha, T) \ll T^{58.05(1 - \alpha)^{3/2}} \log^{15} T.
\eeq

In Section~4 we will deal with the other main ingredient of Linnik's theorem, the famous Deuring--Heilbronn phenomenon.
This asserts that if an $L$-function has a Siegel-zero, then other $L$-functions are free of zeros in some region.

Suppose that $\chi_1$ is a real primitive character $\text{\rm mod}\,q_1$ such that \hbox{$L(1 \! -\! \delta_1, \chi_1)=$} $= 0$ with a real $\delta_1$.
Let $\chi_2$ be an arbitrary primitive character $\text{\rm mod}\,q_2$ such that $L(1 - \delta + it, \chi_2) = 0$,
$\delta_1 < \delta < 1/6$.
(The character $\chi_2$ may be equal to $\chi_1$.)

Suppose $\ve > 0$ arbitrary and
\beq
\label{eq:2.21}
D = [q_1, q_2] (|t| + 1) \geq D_0(\ve).
\eeq
Jutila proved essentially the following version \cite{Jut}:
\beq
\label{eq:2.22}
\delta_1 \geq (1 - 6\delta)D^{-(2 + \ve)\delta/(1 - 6\delta)}/8 \log D.
\eeq

Using Burgess' estimate, Graham \cite{Gra2} improved the exponent $2$ to $3/2$ for bounded~$t$.
Later, using Heath-Brown's estimate (cf. \eqref{eq:3.2} in the next section),
W. Wang showed essentially \cite{Wan}
\beq
\label{eq:2.23}
\delta_1 \geq \frac23 (1 - 6\delta) D^{-(3/2 + \ve)\delta/(1 - 6\delta)} \bigm/ \log D.
\eeq
As in the case of the density theorems, we need a more flexible form of this phenomenon in our application, where apart from the replacement of $2$ by $3/2$, $[q_1, q_2]$ will be replaced by the quantity
\beq
\label{eq:2.24}
q_1^{1/2} q_2^{1/4}(\text{\rm cond }\chi_1 \overline{\chi}_2)^{1/4} \leq [q_1, q_2].
\eeq
(This has no effect in Linnik's theorem, where all quantities can be equal to the same $q$.)
Our version is as follows.

\begin{theorem}
\label{th:5}
Let $\chi_1$ and $\chi_2$ be primitive characters $\text{\rm mod } q_1$ and $q_2$, resp., with
$L(1 - \delta_1, \chi_1) = L(1 - \delta + i\gamma, \chi_2) = 0$, $\chi_1, \delta_1$ real, $\delta_1 < \delta < 1/7$.
Let $k$ be the conductor of $\chi_1 \overline{\chi}_2$.
Let $\ve > 0$ arbitrary,
\beq
\label{eq:2.25}
Y = \bigl(q_1^2 q_2 k(|\gamma| + 2)^2\bigr)^{3/8} \geq Y_0(\ve)
\eeq
sufficiently large.
Then we have
\beq
\label{eq:2.26}
\delta_1 \geq (1 - \ve)(1 - 6\delta)\log 2 \cdot Y^{-(1 + \ve)\delta/(1 - 6\delta)}\bigm/ \log Y.
\eeq
\end{theorem}

\setcounter{section}{3} \setcounter{equation}{0}
\medskip
{\bf 3.}
In the course of our proof we will use four different estimates for the $L$-functions belonging to a character $\text{\rm mod } q$.
Let us define $\varphi = \varphi(\chi) = 1/4$ if $q$ is cube-free, and let $\varphi = 1/3$ otherwise.
Let $k$ be any integer $\geq 3$, $\eta > 0$ be a sufficiently small number.
The first 3 estimates, to be used in Theorems~\ref{th:1}, \ref{th:2} and \ref{th:3}, respectively, are due to Heath-Brown and make crucial use of Burgess' estimates for character sums.
Let $s = \sigma + it$, $\tau = |t| + 2$, then
\beq
\label{eq:3.1}
L\left(\frac12 + it, \chi\right) \ll_\eta (q\tau)^{3/16 + \eta},
\eeq
\beq
\label{eq:3.2}
L(s, \chi) \ll_\eta (q\tau)^{3/8(1 - \sigma) + \eta} \ \ \text{ if } \ 1/2 \leq \sigma \leq 1,
\eeq
\beq
\label{eq:3.3}
L(s, \chi) \ll_{\eta, k} q^{(1 - \sigma)(1 + 1/k) + \eta} \tau \ \ \text{ if } \ 1 - \frac1{k}  \leq \sigma \leq 1;
\eeq
\eqref{eq:3.1} is contained in \cite{Hea1}, \eqref{eq:3.2} is a simple consequence of it by convexity, whereas \eqref{eq:3.3} is Lemma 2.5 in \cite{Hea3}.

The last estimate relies on the bound of Korobov--Vinogradov for which a sharper form is due to K. Ford \cite{For}:
\beq
\label{eq:3.4}
\bigl|\zeta(\sigma + it, u) - u^{-s}\bigr| \leq 76.2~t^{4.45(1 - \sigma)^{3/2}} \log^{2/3} t,
\eeq
for $0 < u \leq 1$, $t \geq 3$, $1/2 \leq \sigma \leq 1$, where $\zeta(s, u)$ is Hurwitz' zeta-function.

Taking into account
\beq
\label{eq:3.5}
L(s, \chi) = \sum_{\ell = 1}^q \frac{\chi(\ell)}{q^{\sigma + it}} \zeta\left(\sigma + it, \frac{\ell}{q}\right) ,
\eeq
the estimate \eqref{eq:3.4} implies that for $t \geq 3$, $1/2 \leq \sigma \leq 1$,
\beq
\label{eq:3.6}
L(s, \chi) \ll q^{1 - \sigma} \bigl(t^{4.45(1 - \sigma)^{3/2}} \log^{2/3} t + \log(q + 1)\bigr).
\eeq
Sometimes, e.g. in Theorem~\ref{th:2}, we will apply a consequence of this, namely
\beq
\label{eq:3.7}
L(s, \chi) \ll (q\tau^\eta)^{1 - \sigma} \log((q + 1)\tau) \ \ \ \text{ for } \ \ \sigma \geq 1 - \eta^2 / 20.
\eeq

We will later make use of the fact that \eqref{eq:3.6} implies that the following region is zero-free.
Let $q \leq M$, $\chi$ primitive $\text{\rm mod } q$, then
\beq
\label{eq:3.8}
L(s, \chi) \neq 0 \ \ \text{ for } \ \sigma \geq 1 - \frac{c_1}{\max(\log M, \log^{2/3} \tau \log^{1/3}_2 \tau)}
\eeq
with the exception of at most one real zero belonging to a real primitive $\chi$ $\text{\rm mod } q \leq M$.

This follows from the note after Satz 6.2 in Chapter VIII of Prachar's book \cite{Pra}, combined with Landau's theorem, in the form given in \cite{Dav}, \S\ 14.

In the proof of Theorems \ref{th:1}--\ref{th:3} we will make use of Linnik's density lemma (see \cite{Pra}, p.~331).

\begin{lemma}
\label{lem:1}
The number of zeros of the function $L(s, \chi)$ $(\chi (\text{\rm mod } q))$ in the square
\beq
\label{eq:3.9}
\alpha \leq \sigma \leq 1, \ \ \ |t - T| \leq (1 - \alpha)/2
\eeq
is
\beq
\label{eq:3.10}
\ll (1 - \alpha) \log \bigl(q(|T| + 2)\bigr) + 1.
\eeq
\end{lemma}

In the proof of Theorem~\ref{th:5}, we will use the following sharper form of Lemma~\ref{lem:1}, a consequence of \eqref{eq:3.6}.

\begin{lemma}
\label{lem:2}
The number of zeros of $L(s, \chi)$ in \eqref{eq:3.9} is
\beq
\label{eq:3.11}
\ll (1 - \alpha)\log q + (1 - \alpha)^{3/2} \log T + \log_2(q T).
\eeq
\end{lemma}

Finally, Hal\'asz's inequality will play a central role in the proof.

\begin{lemma}
\label{lem:3}
Let $f(s, \chi) = \sum\limits_{n = 1}^N a_n \chi(n) n^{-s}$.
Then
\beq
\label{eq:3.12}
\biggl(\sum_{j = 1}^J \bigl|f(s_j \chi_j)\bigr|\biggr)^2 \leq \sum_{n = 1}^N \frac{|a_n|^2}{b_n} \sum_{j, k = 1}^J \eta_j \overline{\eta}_k B(s_j + \overline{s}_k, \chi_j \overline{\chi}_k)
\eeq
where the $\eta_j$ are certain complex numbers of modulus $1$, and
\beq
\label{eq:3.13}
B(s, \chi) = \sum_{n = 1}^\infty b_n \chi(n) n^{-s},
\eeq
where the $b_n$ are arbitrary non-negative numbers such that $b_n > 0$ if $a_n \neq 0$, and $B(s, \chi)$ is absolutely convergent for all pairs $(s_j + \overline{s}_k, \chi_j \overline{\chi}_k)$.
\end{lemma}

This is a modified form of Hal\'asz's inequality given in \cite{Mon1}, Lemma 1.7.
For this form see Jutila \cite{Jut}, Lemma 7.

Clearly, we can suppose that $K \leq M^2$ during our proofs.
Since there is at most one exception, the so-called Siegel zero to \eqref{eq:3.8}, we may suppose that the Siegel zero does not appear among our zeros.
(The upper estimates for $J$ are at least a positive constant in Theorems \ref{th:1}--\ref{th:3}.)
Further, it is enough to show our theorems for non-principal characters and then, additionally, for just the zeta-function.
Thus, we will first show Theorems \ref{th:1}--\ref{th:3} for non-principal characters, and we will then mention the slight modifications which prove them for the zeta-function.

Instead of using pseudocharacters, we will use Graham's approach \cite{Gra2}, in the way performed by Heath-Brown \cite{Hea3}, Lemma 11.1.

In the proof, $\ve$ will denote a sufficiently small positive constant, not necessarily the same as in the formulation of the theorems.

We will use parameters
\beq
\label{eq:3.14}
W = e^w, \ U = e^{\mathcal L}, \ V = UW = e^{\mathcal L + w} = e^v,\ X = e^x
\eeq
to be specified later, with the property
\beq
\label{eq:3.15}
{\mathcal L} \ll_\ve w < {\mathcal L} < x \ll_\ve {\mathcal L}.
\eeq

Following \cite{Hea3}, let us define Graham's weights
\beq
\label{eq:3.16}
\psi_d = \begin{cases}
\mu(d) &\text{for } \ 1 \leq d \leq u,\\
\mu(d)\frac{\log(V/d)}{\log(V/U)} &\text{for } \ U \leq d\leq V,\\
0 &\text{for } \ d \geq V,
\end{cases}
\eeq
and a special case of this $(U = 1)$, namely
\beq
\label{eq:3.17}
\theta_d = \begin{cases}
\mu(d)\frac{\log(W/d)}{\log W} &\text{for } \ 1 \leq d \leq W,\\
0 &\text{for } \ d \geq W.
\end{cases}
\eeq
Set
\beq
\label{eq:3.18}
\Psi(n) = \sum_{d \mid n} \psi_d, \ \ \ \ \vartheta(n) = \sum_{d \mid n} \theta_d.
\eeq
We denote
\beq
\label{eq:3.19}
\alpha = 1 - \delta,
\eeq
and we will choose our parameters in such a way that
\beq
\label{eq:3.20}
x \gg \log M + \log^{2/3} T \log_2^{1/3} T
\eeq
should be satisfied.
In this way, by \eqref{eq:3.8}, we will have, with the exception of the Siegel zero, for any relevant $\varrho_\nu = \beta_\nu + i\gamma_\nu$
\beq
\label{eq:3.21}
\delta \geq 1 - \beta_\nu = \delta_\nu \gg x^{-1}.
\eeq
We take $\chi = \chi_k \neq \chi_0$ with conductor $q = q_k \neq 1$, and with a zero $\varrho_k = \beta_k + i\gamma_k = \varrho = \beta + i\gamma$ of $L(s, \chi)$ and
\beq
\label{eq:3.22}
S(X) = \sum_{n = 1}^\infty \Psi(n)\vartheta(n)\chi(n)n^{-\varrho} e^{-n/X} = \frac{1}{2\pi i} \int\limits_{(1)} L(s + \varrho, \chi) \Gamma(s) X^s F(s + \varrho)ds,
\eeq
where
\beq
\label{eq:3.23}
F(s) = \sum_{i \leq V, \ j \leq W} \psi_i \theta_j \chi([i, j])[i, j]^{-s}.
\eeq
($[i, j]$ always denotes the least common multiple of $i$ and $j$.)

We move the line of integration to $Re\ s = 1 - \beta - h$, where $h$ will be chosen later with $h < 1 - \beta$.
(The integrand is regular between $Re\ s = 1$ and $Re\ s = 1 - \beta - h$.)
Using the estimates $\Gamma(s) \ll e^{-|t|}$ and
\beq
\label{eq:3.24}
F(s + \varrho) \ll \sum_{i \leq V, \ j \leq W} [i, j]^{-1 + h} \ll \sum_{n \leq VW} d^2(n) n^{-1 + h} \ll (VW)^h \mathcal L^3,
\eeq
by \eqref{eq:3.1}--\eqref{eq:3.2} we obtain
\beq
\label{eq:3.25}
\aligned
S(X) &= \int\limits_{(1 - \beta - h)} L(s + \varrho, \chi) F(s + \varrho) \Gamma(s) X^s ds \ll\\
&\ll_\ve \bigl((MT)^{3/8} VWX^{-1}\bigr)^h (UMT)^{\ve^2} X^{1 - \beta}.
\endaligned
\eeq
If
\beq
\label{eq:3.26}
X \gg_\ve \bigl((MT)^{3/8} VW\bigr)^{\frac{1}{1 - \delta/h}} (UMT)^{\frac{2\ve^2}{h(1 - \delta / h)}},
\eeq
then
\beq
\label{eq:3.27}
S(X) = O(\mathcal L^{-1}),
\eeq
where (here and later) the constants implied by the $O$ symbols may depend on~$\ve$.

Taking into account $\Psi(n) = 0$ for $2 \leq n \leq U$, we have
\beq
\label{eq:3.28}
S(U/\mathcal L^2) = e^{-\mathcal L^2 / U} + O \biggl(\sum_{n > U} d(n) e^{-n \mathcal L^2 / U}\biggr) = 1 + O(1/\mathcal L).
\eeq
Thus \eqref{eq:3.27} implies under the condition \eqref{eq:3.26} for $X$
\beq
\label{eq:3.29}
\sum_{n = 1}^\infty \Psi(n) \vartheta(n)\chi(n)n^{-\varrho} \bigl(e^{-n/X} - e^{-n\mathcal L^2 / U}\bigr) = 1 + O(\mathcal L^{-1}).
\eeq
Now we will use Hal\'asz's inequality in the form \eqref{eq:3.12} with
\beq
\label{eq:3.30}
\aligned
a_n &= \Psi(n)\vartheta(n) n^{-1/2} \bigl(e^{-n/X} - e^{-n\mathcal L^2 / U}\bigr),\\
b_n &= \vartheta^2(n) \bigl(e^{-n/X} - e^{-n\mathcal L^2 / U}\bigr), \ \ \ s_j = \varrho_j - 1/2.
\endaligned
\eeq
Using the estimate of Graham \cite{Gra1}, p.~84 that
\beq
\label{eq:3.31}
\sum_{1 < n \leq N} \Psi^2(n) = \begin{cases}
0 & \text{for } \ 1 \leq N \leq U,\\
\frac{N \log(N/U)}{\log^2(V/U)} + O \left(\frac{N}{\log^2(V/U)}\right) &\text{for } \ U \leq N \leq V,\\
\frac{N}{\log(V/U)} + O \left(\frac{N}{\log^2(V/U)}\right) &\text{for } \ N \geq V,
\end{cases}
\eeq
by partial summation we obtain (cf.\ \cite{Hea3}, (11.14)) for $x > v$
\beq
\label{eq:3.32}
\sum_{n = 1}^\infty \frac{|a_n|^2}{b_n} = \sum_{n = 1}^\infty \frac{\Psi^2(n)}{n} \bigl(e^{-n/X} - e^{n \mathcal L^2/U}\bigr)
= \bigl(1 + O(\mathcal L^{-1})\bigr) \frac{2x - \mathcal L - v}{2(v - \mathcal L)}.
\eeq
Any term with $\chi_j \overline{\chi}_k \neq \chi_0$ on the right-hand side of \eqref{eq:3.12} will be, similarly to \eqref{eq:3.25}--\eqref{eq:3.27},
\beq
\label{eq:3.33}
\aligned
&B(s_j + \overline{s}_k, \chi_j \overline{\chi}_k) =\\
&= \sum_{n = 1}^\infty \vartheta^2 (n) \chi_j \overline{\chi}_k(n) = n^{-(\varrho_j + \overline{\varrho}_k - 1)}
\bigl(e^{-n/x} - e^{-n\mathcal L^2 / U}\bigr) \ll\\
&\ll_\ve \bigl((KT)^{3/8} W^2 U^{-1}\bigr)^h \cdot (KT)^{\ve^2} \mathcal L^3 \cdot U^{2\delta} \ll \mathcal L^{-1}
\endaligned
\eeq
if
\beq
\label{eq:3.34}
U \gg_\ve \bigl((KT)^{3/8} W^2\bigr)^{\frac{1}{1 - 2\delta/h}} \cdot (WKT)^{\frac{2\ve^2}{h(1 - 2\delta/h)}}.
\eeq

Let us consider the case $\chi_j \overline{\chi}_k = \chi_{0, q} = \chi_0$ now.
Then, in the case of \eqref{eq:3.34}, we have, similarly,
\beq
\label{eq:3.35}
\aligned
&B(s_j + \overline{s}_k, \chi_0) = \sum_{n = 1}^\infty \vartheta^2(n) \chi_0(n) n^{1 - \varrho_j - \overline{\varrho}_k}
\bigl(e^{-n/X} - e^{-n\mathcal L^2/U}\bigr) = \\
&= \frac{1}{2\pi i} \int\limits_{(1)} L(s + \varrho_j + \overline{\varrho}_k - 1, \chi_0) G_q(s + \varrho_j + \overline{\varrho}_k - 1) \Gamma(s) \bigl(X^s - \bigl(\tfrac{U}{\mathcal L^2}\bigr)^s \bigr) ds = \\
&= \frac{\varphi(q)}{q} G_q(1)\Gamma(2 - \varrho_j - \overline{\varrho}_k)\Bigl(X^{2 - \varrho_j - \overline{\varrho}_k} - \bigl(\tfrac{U}{\mathcal L^2}\bigr)^{2 - \varrho_j - \overline{\varrho}_k} \Bigr) + O(\mathcal L^{-1}),
\endaligned
\eeq
where
\beq
\label{eq:3.36}
G_q(s) = \sum_{\substack{j \leq W,\ k \leq W\\ (j, q) = (k, q) = 1}} \theta_j \theta_k [j, k]^{-s}.
\eeq
The following proposition can easily be proved.

\begin{prop}
$\dfrac{\varphi(q)}{q} G_q(1) \leq \dfrac{1 + C/w}{w}$ \ \ $(w = \log W)$.
\end{prop}

\begin{proof}
We will investigate the finite Dirichlet polynomial $G_q(s)$ for real $s > 1$, $s \to 1$
\beq
\label{eq:3.37}
E(s) = G_q(s) L(s, \chi_0) = \sum_{\substack{n = 1\\ (n, q) = 1}}^\infty \frac{\vartheta^2(n)}{n^s} \leq \sum_{n = 1}^\infty \frac{\vartheta^2(n)}{n^s} \leq \left(1 + \frac{C}{w}\right) w^{-1}\zeta(s)
\eeq
since, applying \eqref{eq:3.31} in the special case of $\Psi = \vartheta$ (that is $U = 1$, $V = W$ in \eqref{eq:3.31} we have for all $y \geq 1$
\beq
\label{eq:3.38}
\sum_{1 < n \leq y} \vartheta^2(n) \leq \frac{y(1 + C/w)}{w} - 1.
\eeq
Taking the limit $s \to 1^+$ in \eqref{eq:3.37}, we obtain the Proposition.
\end{proof}

Let us fix a pair $(\chi_j, \varrho_j)$ $(1 \leq j \leq J)$, and let us consider all zeros $\varrho_{k_\nu}$ $(1 \leq \nu \leq J')$ belonging to the same $L(s, \chi_j)$ (including $\varrho_j$ itself).
Let $\varrho_j = \beta_j + i\gamma_j = 1 - \delta_j + i\gamma_j$, $\varrho_k = \beta_k + i\gamma_k = 1 - \delta_k + i\gamma_k$.
According to \eqref{eq:3.35} and the Proposition, we have (by $\chi_j = \chi_k$)
\beq
\label{eq:3.39}
B(s_j + \overline{s}_k, \chi_j \overline{\chi}_k)\! \ll \! \frac{X^{\delta_j + \delta_k}}{w |\delta_j \! +\! \delta_k \! +\! i(\gamma_j \! -\! \gamma_k)|} \ll \!
\begin{cases}
\frac{X^{2\delta}}{w\cdot \delta} &\text{if }  |\gamma_j - \gamma_k| \leq \delta,\\
\frac{X^{2\delta}}{w\cdot n\delta} &\text{if }  n\delta \leq |\gamma_j\! -\! \gamma_k|\! \leq \! (n\! +\! 1)\delta,
\end{cases}
\eeq
using \eqref{eq:3.21} and the fact that $y^{-1} e^y$ is increasing for $y \geq 1$.
Using Lemma~\ref{lem:1} we see that the number of possible zeros $\varrho_k$ of $L(s, \chi)$ with $n\delta \leq |\gamma_j - \gamma_k| \leq (n + 1)\delta$ and $\delta_k \leq \delta$ is
\beq
\label{eq:3.40}
\ll \delta \log (M(T + n)) + 1.
\eeq
Now, these imply for any fixed $j$
\beq
\label{eq:3.41}
\sum_{\nu = 1}^{J'} \bigl|B(s_j + \overline{s}_{k_\nu}, \chi_0)\bigr| \ll \frac{X^{2\delta}
\delta \log(M(T + J')) + 1}{w\delta} \log J'
\eeq
and so we have in Hal\'asz's Lemma (Lemma~\ref{lem:3}), by \eqref{eq:3.8}, \eqref{eq:3.32} and \eqref{eq:3.41}
\beq
\label{eq:3.42}
J^2 \ll \frac{x}{v - \mathcal L} \left( J \frac{X^{2\delta} \log(MT)}{w} \log^2 J + J^2 \mathcal L^{-1}\right).
\eeq
Hence by \eqref{eq:3.14}--\eqref{eq:3.15}
\beq
\label{eq:3.43}
J \ll_\ve X^{2\delta(1 + \ve)}
\eeq
if
\beq
\label{eq:3.44}
w \gg_\ve \log MT.
\eeq
In order to prove Theorem~\ref{th:1}, we may choose
\beq
\label{eq:3.45}
h = 1/2, \ W = (MT)^\ve, \ U = (KT)^{\frac{3/8 + \ve}{1 - 4\delta}}(MT)^{10\ve},
\eeq
$$
X = (MT)^{\frac{3/8 + 50\ve}{1 - 2\delta}} (KT)^{\frac{3/8 + 50 \ve}{(1 - 2\delta)(1 - 4\delta)}}.
$$
Then all conditions \eqref{eq:3.15}, \eqref{eq:3.20}, \eqref{eq:3.26}, \eqref{eq:3.34}, \eqref{eq:3.44} are satisfied, and this proves \eqref{eq:3.43}, that is, Theorem~\ref{th:1}.

\smallskip
For the proof of Theorem~\ref{th:2} \eqref{eq:2.10A}, we choose
\beq
\label{eq:3.46}
h = \ve, \ W = (MT)^\ve, \ U = K(MT)^{10\ve}, \ X = K(MT)^{3/8 + 100\ve},
\eeq
but instead of \eqref{eq:3.2}, we will use the estimate \eqref{eq:3.7} for the $L$-functions in the estimate of the $B$-functions.
Accordingly, instead of \eqref{eq:3.33}, we have now
\beq
\label{eq:3.47}
B(s_j + s_k, \chi_j \overline{\chi}_k) \ll (KT^\ve W^2 U^{-1})^h \mathcal L^3 U^{2\delta} \cdot(KT)^{\ve^2} \ll \mathcal L^{-1}.
\eeq
Again, all conditions \eqref{eq:3.15}, \eqref{eq:3.20}, \eqref{eq:3.26}, \eqref{eq:3.44} are satisfied and thus \eqref{eq:2.10A} is proved.

\smallskip
The proof of \eqref{eq:2.10B} runs completely analogously.
In this case the role of $K$ and $M$ is `interchanged'.
We choose
\beq
\label{eq:3.47a}
h = \ve, \ W = (MT)^\ve, \ U = (KT)^{50\ve} M^{10\ve}, \ X = M^{1 + 50\ve} (KT)^{3/8 + 100\ve}
\eeq
and use the estimate \eqref{eq:3.7} in the evaluation of $S(X)$ in \eqref{eq:3.23}--\eqref{eq:3.25}
while we use \eqref{eq:3.2} in the estimate of the $B$-functions as in \eqref{eq:3.33}--\eqref{eq:3.34}.

Finally, in case of \eqref{eq:2.10C} we will use both in the evaluation of $S(X)$ and the estimation of the $B$-function the estimate \eqref{eq:3.7}.
According to this we choose in this case
\beq
\label{eq:3.47b}
h = \ve, \ W = (MT)^{\ve/4}, \ U = K^{1 + \ve}(MT)^\ve, \ X = (KM)^{1 + \ve} T^{2\ve}.
\eeq

To prove Theorem~\ref{th:3} we can choose
\beq
\label{eq:3.48}
h = \ve, \ W = (MT)^\ve, U = K^{\varphi + \ve} M^{3\ve} T^{2/\ve}, \ X = (KM)^{\varphi + 10\ve} T^{10/\ve}.
\eeq
Applying the estimate \eqref{eq:3.3} with $k = [\ve^{-1}]$ and $\eta = \ve^2 / 2$, we obtain, instead of \eqref{eq:3.25} and \eqref{eq:3.33}, the estimates
\beq
\label{eq:3.49}
S(X) \ll \bigl(M^{(1 + 2\ve)\varphi} T^{1/\ve} VWX^{-1}\bigr)^{\ve} (UMT)^{\ve^2} X^{1 - \beta} \ll \mathcal L^{-1}
\eeq
and
\beq
\label{eq:3.50}
B(s_j + \overline{s}_k, \chi_j \overline{\chi}_k) \ll \bigl(K^{(1 + 2\ve)\varphi} T^{1/\ve} W^2 U^{-1}\bigr)^\ve (UKT)^{\ve^2} U^{2\delta} \ll \mathcal L^{-1}.
\eeq
Since the conditions \eqref{eq:3.15}, \eqref{eq:3.20}, \eqref{eq:3.44} are again satisfied, Theorem~\ref{th:3} is also proved.

\smallskip
In the case of the Riemann zeta-function, Theorem~\ref{th:3} is clearly much weaker than any of Theorems \ref{th:1} and \ref{th:2}.
Theorem~\ref{th:2} clearly follows from \eqref{eq:2.16} for the zeta-function.

On the other hand, in case of the zeta-function Theorem~\ref{th:1}, that is \eqref{eq:2.9} follows for $\alpha \geq 4/5$ from Theorem 11.4 of Ivi\'c \cite{Ivi} since
\beq
\label{eq:3.51}
\frac{3}{2\alpha} < \frac{3}{2(4\alpha - 3)}.
\eeq

\setcounter{section}{4} \setcounter{equation}{0}
\subsubsection*{Proof of Theorem~\ref{th:5}}

\indent{\bf 4.}
Following Jutila \cite{Jut}, we will use the idea of Selberg \cite{MS} to apply pseudocharacters
\beq
\label{eq:4.1}
f_r(n) = f((r,n))
\eeq
with multiplicative arithmetic functions $f$ where $(a, b)$ denotes the greatest common divisor of $a$ and $b$.

Let us use the abbreviation
\beq
\label{eq:4.2}
f_r f_{r'}(n) = f_r(n) f_{r'}(n).
\eeq
For the exceptional real non-principal character $\chi_1$ let
\beq
\label{eq:4.3}
a_n = \sum_{d \mid n} \chi_1 (d) = \prod_{p^\alpha \| n} \bigl(1 + \chi_1(p) + \dots + \chi_1^\alpha(p)\bigr) \geq 0.
\eeq
If $n$ is square-free, then $a_n\! =\! 0$ if there exists a prime divisor $p$ of $n$ with $\chi_1(p) \! =\! -1$.
If $n$ is square-free and $\chi_1(p) \!=\! 1$ for all $p\! \mid \!n$, then
$a_n \! =\! 2^{\omega(n)}$, where $\omega(n)$ is the number of prime factors of the square-free number $n$ \,\hbox{$(a_1 \!=\! 1)$}.

Both $\chi_1(n)$ and $\chi_2(n)$ can be considered as characters $\text{\rm mod } q = [q_1, q_2]$.
Let $\chi_0$ be the principal character $\text{\rm mod } q$, $\mu(n)$ the M\"obius function.
Let $\sum'$ denote summation over all square-free numbers coprime to~$q$.
Let $S = \sum_{r \leq R} \hspace*{-1pt}{}^{\displaystyle\prime} a_r r^{-1}$ with the parameter $R$ to be chosen later.

In the course of proof we will need the following lemmas:

\begin{lemma}
\label{lem:uj1}
Let $\chi$ be a Dirichlet character, $f$ a multiplicative function, $r$ and $r'$ square-free numbers such that $\chi_1(p) = 1$ for all prime divisors of $rr'$, and define for $Re\ s > 1$
\beq
\label{eq:4.4}
G_{r,r'}(s, \chi) = \sum_{n = 1}^\infty \mu^2(n) a_n \chi(n)f_r f_{r'} (n) n^{-s}.
\eeq
Then
\beq
\label{eq:4.5}
G_{r,r'}(s, \chi) = L(s, \chi) L(s, \chi\chi_1) P_{r,r'}(s, \chi) Q(s, \chi),
\eeq
where
\beq
\label{eq:4.6}
\aligned
P_{r,r'} (s, \chi) &= \prod_{\substack{p \mid rr'\\ p\nmid (r,r')}}\! \left(\! 1\! + \frac{2\chi(p)f(p)}{p^s}\! \right) \!
\prod_{p\mid (r, r')}\! \left(\! 1\! + \frac{2\chi(p) f^2(p)}{p^s}\!\right)\! \prod_{p\mid rr'} \left(\! 1 \! + \frac{2\chi(p)}{p^s}\!\right)^{\!-1}\! ,\\
Q(s,\chi) &= \prod_{\chi_1(p) = 1} \left(1 - \frac{\chi(p)}{p^s} \right)^2\left(1 + 2 \frac{\chi(p)}{p^s}\right) \prod_{\chi_1(p) = -1} \left(1 - \frac{\chi^2(p)}{p^{2s}}\right).
\endaligned
\eeq
\end{lemma}

\begin{proof}
This is Lemma 9 of Jutila \cite{Jut}.
\end{proof}

\begin{lemma}
\label{lem:uj2}
In the preceding lemma, choose
$$
f(n) = \mu(n)2^{-\omega(n)}n,
$$
and suppose also that $L(\chi_1, \beta_1) = 0$, where $\beta_1 = 1 - \delta_1$ is a real number satisfying $3/4 < \beta_1 < 1$.
Then for the sum
\beq
\label{eq:4.7}
T = \sum_{n = 1}^\infty \hspace*{-1pt}{}^{\displaystyle\prime} a_n e^{-n/Y} n^{-\beta_1} \biggl(\sum_{r \leq R}\hspace*{-1pt}{}^{\displaystyle\prime} a_r f_r(n)r^{-1} \biggr)^2
\eeq
we have for every $R$ the asymptotic formula
\beq
\label{eq:4.8}
T = \frac{\varphi(q)}{q} Q(1, \chi_0) L(1, \chi_1) \Gamma(\delta_1) Y^{\delta_1} S + O_\ve \bigl(Rq_1^{3/16 + \ve} Y^{1/2 - \beta_1 + \ve}\bigr).
\eeq
\end{lemma}

\begin{proof}
This is a sharpened form of Lemma~10 of \cite{Jut}, with the only change that on the line $\sigma = 1/2$ we use the estimate \eqref{eq:3.1} of Heath-Brown for $L(s, \chi_1)$.
This form of Lemma~\ref{lem:uj2} appears also as Lemma~12 of \cite{Wan}.
Therefore we can replace $q^{1/4}$ of \cite{Jut} by $q_1^{3/16}$.
\end{proof}

\begin{lemma}
\label{lem:uj3}
For every $R_0$ there exists an $R \in [R_0, 2R_0]$ such that
\beq
\label{eq:4.9}
S = S(R) \geq \frac{\varphi(q)}{q} Q(1, \chi_0) \mathcal L(1, \chi_1) \delta_1^{-1} + O_\ve (R^{-1/2 + \ve} q_1^{3/16 + \ve}).
\eeq
\end{lemma}

\begin{proof}
This is also a sharpened form of Lemma 11 of \cite{Jut} at least for the suitably chosen~$R$.
Here we must further modify the proof.
The generating function of $\mu^2(n) a_n \chi_0(n)$ is $F(s) = L(s, \chi_1) L(s, \chi_0) Q(s, \chi_0)$.

Hence we have for all $R$
\beq
\label{eq:4.10}
S \geq R^{-\delta_1} \sum_{r \leq R}\hspace*{-1pt}{}^{\displaystyle\prime} \frac{a_r}{r^{\beta_1}} = \frac{R^{-\delta_1}}{2\pi i} \int\limits_{a - iRq}^{a + iRq} \frac{F(s + \beta_1)R^s}{s} ds + O(R^{-1/2}),
\eeq
where $a\! =\! \delta_1 \! +\! 1/\!\log (qR)$.
Moving the line of integration to the line $Re(s\! +\! \beta_1) = 1/2 + \ve$, we get a pole at $s = \delta_1$ with residue
\beq
\label{eq:4.11}
R^{-\delta_1} \cdot \frac{R^{\delta_1}}{\delta_1} \cdot Q(1, \chi_0) \frac{\varphi(q)}{q} L(1, \chi_1)
\eeq
and
\beq
\label{eq:4.12}
\frac{1}{2\pi i} \int\limits^{\delta_1 - \frac12 + \ve + iRq}_{\delta_1 - \frac12 + \ve - iRq} \frac{F(s + \beta_1)R^{s - \delta_1}}{s} ds + O(R^{-1/2}) = I(R) + O(R^{-1/2}).
\eeq
The average of the real integral $I(R)$ is clearly
\beq
\label{eq:4.13}
E(R_0) = \frac1{R_0} \int\limits_{R_0}^{2R_0} I(R)dR = \frac{1}{2\pi i}
\int\limits^{\delta_1 - \frac12 + \ve + iRq}_{\delta_1 - \frac12 + \ve - iRq} \frac{F(s + \beta_1)R_0^{s - \delta_1}}{s(s + \beta_1)} (2^{s + \beta_1} - 1)ds.
\eeq
Using again the estimate \eqref{eq:3.1} of Heath-Brown for $L(s, \chi_1)$, we obtain for $E(R_0)$ the estimate given in the error term of \eqref{eq:4.9}.
\end{proof}

\begin{lemma}
\label{lem:uj4}
Let $\beta_1$ be as in the preceding lemmas, and suppose also that $L(\varrho, \chi) = 0$, where $\chi$ is a character $(\text{\rm mod } q_2)$, and $\varrho = \beta + i\gamma$, $3/4 < \beta < \beta_1$.
Put $D = \bigl(q_2 k (|\gamma| + 2)^2\bigr)^{3/8}$.
Then in the case $\chi \neq \chi_0, \chi_1$ we have, for the quantity $T$ defined by \eqref{eq:4.7}, the estimates
\beq
\label{eq:4.14}
T \geq S^2(1 + Y^{(1 + \ve)(\beta - \beta_1)}) + O_\ve(RD^{1/2 + \ve} Y^{1/2 - \beta_1 + \ve})
\eeq
for all $R$.
If $\chi = \chi_0$ or $\chi_1$, then for every $R_0$ there exists an $R \in [R_0, 2R_0]$ such that either
\beq
\label{eq:4.15}
T \geq S^2\bigl(1 + (1 - \ve) Y^{(1 + \ve)(\beta - \beta_1)}\bigr) + O_\ve(R_0 D^{1/2 + \ve} Y^{1/2 - \beta_1 + \ve})
\eeq
or
\beq
\label{eq:4.16}
\delta_1 \geq \ve Y^{\beta - 1}|\Gamma(1 - \varrho)|^{-1} \bigl\{1 + O_\ve(R_0^{-1/2 + \ve} q_1^{3/16 + \ve})\bigr\}.
\eeq
\end{lemma}

\begin{proof}
This is again a sharpened form of Lemma~12 of \cite{Jut} which can be proved using the estimate \eqref{eq:3.1} and our Lemma~\ref{lem:uj3} in place of Lemma~11 of \cite{Jut}.
\end{proof}

Another minor change in the proof is that in (5.7)--(5.8) of \cite{Jut} we will replace the factor $1/2$ by $1 - \ve$, and, accordingly, for the $\theta$ in (5.9) of \cite{Jut} we have the inequality $(2 - \ve)^{-1} < \theta < \ve^{-1}$ in place of $2/3 < \theta < 2$.
This makes \eqref{eq:4.15} slightly stronger and the less crucial \eqref{eq:4.16} weaker.

We remark that in the formula before (5.7) of \cite{Jut} on the right side of the inequality a factor $S$ is missing from \hbox{$(\varphi(q) / q) Q(1, \chi_0)L(1,\chi_1)|\Gamma(1\! -\! \varrho)|Y^{1 - \beta}$} by a misprint (see the corresponding formula for $T_\chi$ above it).

\begin{proof}[Proof of Theorem~\ref{th:5}]
The proof follows that of Jutila \cite{Jut}, with slight changes, so we will be brief.
For all $R_0$ we can choose a fixed value of $R$ such that Lemmas \ref{lem:uj3} and \ref{lem:uj4} should hold with the same~$R$.
Since $R \in [R_0, 2 R_0]$ it is irrelevant whether we write $R$ or $R_0$ in the error terms.
We can suppose $\delta_1 \ll (\log Y)^{-1}$ and $\delta_1 \ll (\log q_1)^{-1}$, otherwise \eqref{eq:2.26} holds.
The choices of $R_0$ and $Y$ will imply $\log(R_0 Y) \ll \log D$ due to $q_1 \leq q_2 k$.
Let us consider first the case $\chi \neq \chi_0 , \chi_1$.
Then the comparison of Lemmas~\ref{lem:uj2} and \ref{lem:uj4}, namely \eqref{eq:4.8} and \eqref{eq:4.14}, imply, with the notation
\beq
\label{eq:4.17}
B = \frac{\varphi(q)}{q} Q(1, \chi_0) \mathcal L(1, \chi_1),
\eeq
the inequality
\beq
\label{eq:4.18}
\aligned
B \Gamma(\delta_1)Y^{\delta_1} S &\geq S^2 (1 + Y^{(1 + \ve)(\beta - \beta_1)}) +\\
&\quad + O_\ve (R_0 q_1^{3/16 + \ve} Y^{1/2 - \beta_1 + \ve}) + O_\ve(R_0 D^{1/2 + \ve} Y^{1/2 - \beta_1 + \ve}).
\endaligned
\eeq
The first error term can be neglected, since it is inferior to the second, in view of $D \geq (q_2 k)^{3/8} \geq q_1^{3/8}$.
Cancelling this inequality and replacing $S$ by the estimate from Lemma~\ref{lem:uj3},
by $S \geq a_1 = 1$ we obtain
\beq
\label{eq:4.19}
\aligned
B \Gamma(\delta_1)Y^{\delta_1} &\geq B\delta_1^{-1} (1 + Y^{(1 + \ve)(\beta - \beta_1)}) +\\
&\quad + O_\ve (R_0 D^{1/2 + \ve} Y^{1/2 - \beta_1 + \ve}) + O_\ve(R_0^{-1/2 + \ve} q_1^{3/16 + \ve}).
\endaligned
\eeq
Here $\Gamma(\delta_1)\delta_1 = 1 + O(\delta_1)$, and from the integral representation \eqref{eq:4.10} we obtain an $R \in [q_1, 2q_1]$ such that by $\delta_1 \ll (\log q_1)^{-1}$,
\beq
\label{eq:4.20}
1 \ll R^{-\delta_1} \sum_{n \leq R}\hspace*{-1pt}{}^{\displaystyle\prime} \frac{a_r}{r} = B \delta_1^{-1} + O(q_1^{-1/2}) + O(q_1^{-1/16 + 2\ve}).
\eeq
Thus we may divide by $B\Gamma(\delta_1) \gg 1$ to obtain
\beq
\label{eq:4.21}
Y^{\delta_1} \geq 1 + Y^{-(1 + \ve)\delta} + O(\delta_1) + O_\ve (R_0 D^{1/2 + \ve} Y^{1/2 - \beta_1 + \ve}) + O_\ve(R_0^{-1/2 + \ve}q_1^{3/16 + \ve}).
\eeq
Now we will choose $R_0$ and $Y$ in such a way that the last two error terms on the right side should be lower order of magnitude than $Y^{-(1 + \ve)\delta}$.
Let
\beq
\label{eq:4.22}
Y = (Dq_1^{3/4})^{\frac{1}{1 - 6\delta} + \ve_1}, \ \ R_0 = q_1^{3/8 + \ve_2} (Dq_1^{3/4})^{\frac{2\delta}{1 - 6\delta} + \ve_2},
\eeq
where $\ve_1$ and $\ve_2$ are properly chosen small numbers depending on~$\ve$.
(That is, $\ve_1(\ve)$ and $\ve_2(\ve) \to 0$ as $\ve \to 0$.)
With the above choice from \eqref{eq:4.21} we obtain the estimate
\beq
\label{eq:4.23}
e^{\delta_1 \log Y} - 1 + O(\delta_1) \geq (1 - \ve) Y^{-(1 + \ve)\delta}.
\eeq
Now, if $u = \delta_1 \log Y \geq \log 2$, then \eqref{eq:2.26} clearly holds.
If $u < \log 2$, then $e^u - 1 \leq u / \log 2$ and so \eqref{eq:4.23} yields
\beq
\label{eq:4.24}
\delta_1 \log Y / \log 2 \geq (1 - 2\ve) Y^{-(1 + \ve)\delta},
\eeq
which proves Theorem~\ref{th:5}, when $\chi \neq \chi_0$ or $\chi_1$.
\end{proof}

If $\chi = \chi_0$ or $\chi_1$ and \eqref{eq:4.15} holds, then the same argument as above applies and we obtain \eqref{eq:4.24} with $1 - 2\ve$ replaced by $1 - 3\ve$.

If $\chi = \chi_0$ or $\chi_1$ and \eqref{eq:4.16} holds, then \eqref{eq:4.16} implies the estimate
\beq
\label{eq:4.25}
\delta_1 > \ve(1 - \ve)Y^{-\delta} \left|\frac{\Gamma(1 - \varrho)^{-1}}{1 - \varrho}\right| |1 - \varrho|,
\eeq
which clearly proves our theorem if
\beq
\label{eq:4.26}
|1 - \varrho| > \ve^{-1} / \log Y.
\eeq
This is trivially true for $\chi_0$.
If $\chi = \chi_1$, then $Y \geq q_1^{9/8}(|t| + 1)^{3/4} > q_1$.
If
\beq
\label{eq:4.27}
|1 - \varrho| \leq \ve^{-1} / \log Y,
\eeq
then we can apply Lemma 8.4 of Heath-Brown \cite{Hea3}.
This asserts, with our notation, that
\beq
\label{eq:4.28}
\frac1{\delta_1 \log Y} \leq \frac1{\delta_1 \log q_1} \leq e^{(2/3)\delta \log q_1} = q_1^{2\delta/3} \leq Y^{2\delta/3}
\eeq
if $q_1 > q_0(\ve)$, and therefore immediately proves \eqref{eq:2.26}.

If $q_1 \leq q_0(\ve)$, then all zeros of any $L$-functions $\text{\rm mod }  q \leq q_0(\ve)$
are at a distance at least $d_0(\ve)$ from~$1$.
Therefore \eqref{eq:4.26} will be true if
\beq
\label{eq:4.29}
\log Y > (\ve d_0(\ve))^{-1} \Longleftrightarrow Y > Y_0(\ve) := e^{(\ve d_0(\ve))^{-1}}.
\eeq

\noindent
J\'anos Pintz\\
R\'enyi Mathematical Institute\\
of the Hungarian Academy of Sciences\\
Budapest, Re\'altanoda u. 13--15\\
H-1053 Hungary\\
e-mail: pintz.janos@renyi.mta.hu

\end{document}